\documentclass[10pt]{amsart}

\usepackage{amsmath,amssymb,amsthm}
\usepackage{a4wide}
\usepackage{mathtools}
\newlength{\arrowbla}

\usepackage{graphicx}
\usepackage{xypic}
\entrymodifiers={+!!<0pt,\fontdimen22\textfont2>}
\usepackage[all]{xy}
\usepackage{chngpage}
\usepackage{hyperref}

\newtheorem{lemma}{Lemma}[section]
\newtheorem{theorem}[lemma]{Theorem}

\newtheorem{definition}[lemma]{Definition}

\newtheorem{proposition}[lemma]{Proposition}

\newcommand{\op}{\mathrm{op}}
\newcommand{\niceo}{\mathcal{O}}
\newcommand{\nices}{\mathcal{S}}
\newcommand{\p}{\niceo\Sigma}
\newcommand{\pp}{\mathcal{P}}
\newcommand{\sets}{\mathrm{Set}}
\newcommand{\ch}{\mathbf{Ch}}

\newcommand{\card}[1]{|#1|}
\newcommand{\id}{\mathrm{id}}
\newcommand{\sgn}{\mathrm{sgn}}
\newcommand{\ojoin}{J^\niceo}
\newcommand{\symjoin}{J^\Sigma}
\newcommand{\chsusp}{\sigma}
\newcommand{\fraka}{\mathfrak{a}}
\newcommand{\frakj}{\mathfrak{j}}
\newcommand{\ord}{\overline}

\newcommand*{\myrightarrow}[1]{\xrightarrow{\mathmakebox[\arrowbla]{#1}}}

\newcommand{\removed}[1]{}
\newcommand{\michal}[1]{}

\begin{document}
%===============================================================
\title{The Symmetric Join Operad}

%================================
\author[Micha{\l} Adamaszek]{Micha{\l} Adamaszek}
\address{Mathematics Institute and DIMAP,
      \newline University of Warwick, Coventry, CV4 7AL, UK}
\email{aszek@mimuw.edu.pl}
\thanks{The first author was supported by the Centre for Discrete
        Mathematics and its Applications (DIMAP), EPSRC award EP/D063191/1.}
%===============================

%===============================
\author[John D.S. Jones]{John D.S. Jones}
\address{Mathematics Institute,
      \newline University of Warwick, Coventry, CV4 7AL, UK}
\email{J.D.S.Jones@warwick.ac.uk}
%===============================


% \date{Draft. \today}

%===============================================================
\begin{abstract}
The join operad arises from the combinatorial study of the iterated join of simplices. We study a suitable simplicial version of this operad which includes the symmetries given by permutations of the factors of the join. From this combinatorics we construct an $E_\infty$-operad which coacts naturally on the chains of a simplicial set.
\end{abstract}
\maketitle

%========================================================================
%========================================================================
\section{Introduction}

Right from the outset we establish the following convention: {\bf we index simplices by the number of vertices, rather than by dimension, that is we write $\Delta_k$ for a simplex with $k$ vertices and therefore dimension $k-1$.} 

Let $X$ and $Y$ be topological spaces.  Then $X\ast Y$, the {\em join} of $X$ and $Y$, is a quotient space of $I \times X \times Y$ and so there is a canonical map
$$
I \times X \times Y \to  X \ast Y.
$$
More generally, there are canonical maps 
$$
\Delta_k \times X_1 \times \dots \times X_k \to X_1 \ast \dots \ast X_k.
$$
The join of two simplices is a simplex and these canonical maps make the sequence of spaces $\Delta_k$ into a topological operad.

If we look for a version of  these canonical maps at the level of simplicial sets we are led, quite naturally, to two fundamental points.
\begin{itemize}
\item
The simplest way to construct such maps for simplicial sets uses exactly the same geometric constructions with simplices which are the basis of Steenrod's construction \cite{Steen} of the cup-i products at the cochain level.
\item
While the join operation of spaces is symmetric it is not symmetric at the level of simplicial sets.  The basic point is that if $A$ and $B$ are geometric simplices and if we order the vertices of $A$ and $B$, then the natural orderings of the vertices of $A \ast B$ and $B \ast A$ are not the same.
\end{itemize}

In the theory of classical simplicial complexes simplices are uniquely determined by their vertices and so as usual we identify a simplex with its set of vertices.  Then, assuming $K$ and $L$ are classical simplicial complexes, the simplices of $K \ast L$ are given by $\sigma\sqcup\tau$ where $\sigma$ is a simplex in $K$ or else it is empty and $\tau$ is a simplex in $L$ or else it is empty.  The simplest way of dealing with this is to add to both $K$ and $L$ a simplex of dimension $-1$ with the empty set as its set of vertices.  Passing to simplicial sets this amounts to  working with simplicial sets augmented by adjoining a set of simplices which have \lq no vertices\rq\ or dimension $-1$. In terms of the category $\Delta$ of ordered sets and order preserving maps used in the theory of simplicial sets this corresponds to replacing $\Delta$ by the category  $\niceo$ obtained by adjoining to $\Delta$ an additional initial object and working with $\niceo$-sets, also known as augmented simplicial sets, rather than $\Delta$-sets.  The category of $\niceo$-sets is the natural context for the join operation, see \S \ref{section:2} for the details.  

However the join of $\niceo$-sets is not symmetric.  What we lack is an action of the symmetric group $\Sigma_n$ on the set of simplices with $n$ vertices which `behaves like' reordering the vertices of an ordered simplex.  This brings us to the framework of $\niceo\Sigma$-sets, an augmented version of $\Delta\Sigma$-sets, one of the basic examples of crossed simplicial groups studied by Fiedorowicz and Loday \cite{FL}. Any augmented simplicial set determines a free $\niceo\Sigma$-set, and the category of $\niceo$-sets is equivalent to the category of $\niceo\Sigma$-sets in a suitable homotopy theoretic sense, see \ref{subsect:homotopy}. This leads to the symmetric join $\symjoin$ of $\niceo\Sigma$-sets and symmetric versions of the canonical maps
$$
\niceo\Sigma_n\times X_1\times\cdots\times X_n\to\symjoin(X_1,\ldots,X_n)
$$
where $\niceo\Sigma_n$ is the $\niceo\Sigma$-set corresponding to the augmented simplex $\Delta_n$ .

The join of $\niceo\Sigma_n$ and $\niceo\Sigma_m$ is $\niceo\Sigma_{n+m}$ and these constructions produce an operad $\{\niceo\Sigma_n\}_{n\geq 0}$ in the category of $\niceo\Sigma$-sets.  This is the fundamental object in this paper.  The other operads we consider are derived from this one by forming pairs, applying the forgetful functor to simplicial sets, and applying chains.   Combining the canonical maps with the diagonal $ X \to X^n$ we get a canonical `coaction' 
$$
\niceo\Sigma_n\times X\to\symjoin(X,\ldots,X)
$$
on any $\niceo\Sigma$-set $X$.

Having sorted out the combinatorics of the symmetric join operad of $\niceo\Sigma$-sets we convert it into a chain level $E_\infty$-operad $\{\frakj(n)\}_{n\geq 0}$, see Definition \ref{defn:operad}. The canonical maps  yield, after some manipulation, chain level coaction
$$
\frakj(n)\otimes C_*(Y)\to C_*(Y)^{\otimes n}.
$$
for any simplicial set $Y$, making $C_*(Y)$ into a coalgebra over $\frakj$. 

The cooperations in $C_*(Y)$ induced by this coaction include the classical Alexander-Whitney coproduct and higher cup-i coproducts whose duals in $C^*(Y)$ are the standard cup product and higher cup-i products.   This expresses in a very precise sense the relation between Steenrod's construction of the cup-i products and the combinatorics of the join of simplices which was the starting point for this work.

This is not at all the first construction of a chain $E_\infty$-operad which endows the chain complex $C_*(Y)$ (resp. $C^*(Y)$) with a natural structure of an $E_\infty$-coalgebra (resp. $E_\infty$-algebra). In \cite{McSmith03} McClure and Smith constructed the sequence operad and its action on $C^*(Y)$.  We will compare the symmetric join operad with the sequence operad in Section \ref{subsect:constr}. Berger and Fresse \cite{BerFres04} analyse the Barratt-Eccles operad, which is the chain version of a simplicial operad obtained from the bar constructions of symmetric groups. Another such operad is the condensation of the multi-coloured lattice path operad of Batanin and Berger \cite{BatBer}. It contains the sequence operad as a suboperad. There is also a significant amount of unpublished work  due, independently, to Jim Milgram and Ezra Getzler to be acknowledged.  These constructions are significant since by \cite{Man01,Man06}, under suitable assumptions on $Y$, an $E_\infty$-algebra structure on $C^*(Y)$ determines the homotopy type of $Y$. 

Perhaps the main theme in this paper is to understand better the relation between chain cooperations (or cochain operations) and the combinatorics of joins.  Another important theme is the idea that $\niceo\Sigma$-sets give a conceptual way of encoding the relations between reorderings of the vertices of a simplex and its faces and degeneracies and therefore form the most natural setting for our general constructions. Finally, one of our aims is to do as much as possible in the context of simplicial geometry and translate this into a coaction on chains in the last possible moment.

This paper is set out as follows:  \S\ref{section:2} contains the background on the necessary modifications of simplicial sets, that is $\niceo$-sets and $\niceo\Sigma$-sets which are the natural setting for our constructions;  \S\ref{section:3} describes in detail the construction of the symmetric join operad; in \S\ref{section:4} we describe the manipulations required to produce and $E_\infty$-operad and its coaction on $C_*(Y)$ where $Y$ is a simplicial set; finally \S\ref{sect:appendix} contains some deferred proofs.

%========================================================================
%========================================================================
\section{Background}
\label{section:2}
%========================================================================
%========================================================================
\subsection{Indexing categories.}

Let $\Delta$ denote the usual simplicial category, with objects $[n]=\{0,\ldots,n\}$ for $n\geq 0$ and order-preserving maps as morphisms. For $0\leq i\leq n$ let $\delta_i: [n-1]\to [n]$ be the injection whose image does not contain $i$ and let $\sigma_i: [n+1]\to [n]$ be the surjection which maps $i$ and $i+1$ to the same value.

Let $\niceo$ be the ordinal category, whose objects are the sets $\ord{n}=\{1,\ldots,n\}$ for $n\geq 0$ where $\ord{0}=\emptyset$ and whose morphisms are the order-preserving maps. For $1\leq i\leq n$ let $\delta_i: \ord{n-1}\to \ord{n}$ be the injection whose image does not contain $i$ and let $\sigma_i:\ord{n+1}\to \ord{n}$ be the surjection which maps $i$ and $i+1$ to the same value. For each $n$ there is a unique morphism $\ord{0}\to\ord{n}$ and there are no morphisms $\ord{n}\to\ord{0}$ if $n>0$.

The assignment $[n]\to\ord{n+1}$ defines an inclusion $\Delta\hookrightarrow\niceo$. It sends the morphisms $\delta_i$, $\sigma_i$ of $\Delta$ to $\delta_{i+1}$, $\sigma_{i+1}$ of $\niceo$, respectively.

Our main tool is the category $\niceo\Sigma$, see \cite{FL,Pirashvili}. Its objects are the same as the objects of $\niceo$. A morphism $f\in\niceo\Sigma(\ord{n},\ord{m})$ is, by definition, a map of sets together with a complete order on each of the sets $f^{-1}(i)$ for $i\in\ord{m}$.
%  Equivalently, it is a map of sets and a partial ordering of $\ord{n}$ which restricts to a complete order on each fibre $f^{-1}(i)$ and in which elements from two different fibres are incomparable. This ordering will be denoted $<_f$ and the underlying set map of a morphism $f$ will also be denoted $f$. The composition $gf$ of morphisms in $\niceo\Sigma$ is given by composing the underlying set maps and ordering each fibre $(gf)^{-1}(i)$ as an ordered union over $g^{-1}(i)$ of the ordered sets $f^{-1}(k)$, $k\in g^{-1}(i)$. Formally
% \begin{equation*}
% x<_{gf} y \ \textrm{if}\ f(x)<_g f(y) \ \textrm{or}\ x<_f y.
% \end{equation*}
% This composition is clearly associative. 
There is a forgetful functor $\niceo\Sigma\to\sets$ which discards the ordering and remembers the underlying set map.

It follows that an \emph{injective} set map $f:\ord{n}\to\ord{m}$ determines a unique morphism in $\niceo\Sigma$.
% whose underlying set map is $f$.
In particular we have inclusions $\Sigma_n\subseteq\p(\ord{n},\ord{n})$ for all $n$. Moreover, there is a canonical embedding of categories 
$$\niceo\hookrightarrow \p.$$
On the objects it is the identity and it takes an order-preserving map to the morphism with the same underlying set map and with each fibre ordered according to the natural ordering of the source.

The category $\p$ has another description. First note that any morphism $\phi\in\p(\ord{n},\ord{m})$ has a unique decomposition as $\phi=f \circ\pi$ where $f\in\niceo(\ord{n},\ord{m})\subseteq\p(\ord{n},\ord{m})$ and $\pi\in\Sigma_n\subseteq\p(\ord{n},\ord{n})$. We can therefore identify elements of $\p(\ord{n},\ord{m})$ with pairs $(f,\pi)\in\niceo(\ord{n},\ord{m})\times \Sigma_n$. Let us express the composition of morphisms in this representation. First, for any permutation $\pi\in\Sigma_n$ and an order-preserving map $g\in\niceo(\ord{k},\ord{n})$, define a permutation $g^*\pi\in\Sigma_k$ and a map $\pi_*g\in\niceo(\ord{k},\ord{n})$ as the unique pair with the following two properties:
\begin{itemize}
\item the diagram
\begin{equation*}
\xymatrix{
\ord{k}\ar[r]^g\ar[d]_{g^*\pi} & \ord{n}\ar[d]^\pi\\
\ord{k}\ar[r]_{\pi_*g} & \ord{n}
}
\end{equation*}
commutes
\item the permutation $g^*\pi$ is order-preserving on each of the fibres $g^{-1}(i)$.
\end{itemize}
A quick calculation now shows that the composition of morphisms in $\p$ is defined as follows: if $(f,\pi)\in\p(\ord{n},\ord{m})$ and $(g,\sigma)\in\p(\ord{k},\ord{n})$ then
\begin{equation}
(f,\pi)\circ(g,\sigma) = (f\circ(\pi_*g),(g^*\pi)\circ\sigma)\in\p(\ord{k},\ord{m}).
\end{equation}
In this representation the inclusion $\niceo(\ord{n},\ord{m})\hookrightarrow\p(\ord{n},\ord{m})$ is given by $f\to(f,\id)$, the forgetful functor $\p\to\sets$ is given by $(f,\pi)\to f\pi$ and the inclusion $\Sigma_n\subseteq\p(\ord{n},\ord{n})$ is $\pi\to(\id,\pi)$.

%========================================================================
%========================================================================
\subsection{Some notation.}
\label{subsect:extranotation}
Let $\pp$ denote any of the categories $\niceo$, $\p$ or $\sets$.  For any subset $I\subseteq \ord{m}$ let $i_I:\ord{\card{I}}\to\ord{m}$ be the unique morphism in $\pp$ determined by the order-preserving injection of the set $\ord{\card{I}}$ into $\ord{m}$ with image $I$. For any morphism $f\in\pp(\ord{k},\ord{m})$ and any subset $I\subseteq\ord{m}$ we define $f_I:\ord{\card{f^{-1}(I)}}\to\ord{m}$ and $f^I:\ord{\card{f^{-1}(I)}}\to\ord{\card{I}}$ as the unique morphisms in $\pp$ which make the following diagram commute:
\[
\xymatrix{
\ord{\card{f^{-1}(I)}} \ar[d]_{i_{f^{-1}(I)}} \ar[r]^{f^I} \ar[dr]^{f_I}  & \ord{\card{I}} \ar[d]^{i_I} \\
\ord{k} \ar[r]^f                  & \ord{m} 
}
\]
Intuitively, $f_I$ is the morphism $f$ restricted to the preimage of $I$ and $f^I$ is obtained by further restricting the target to $I$. Note that $i_I=\id_I$.

Now every morphism $f\in\pp(\ord{k},\ord{n})$ determines a decomposition of $\ord{k}$ into $n$ blocks $\{f^{-1}(i)\}_{i=1}^n$ of sizes $a_i=\card{f^{-1}(i)}$. Given a sequence of morphisms defined on these blocks, 
$$g_i\in\pp(\ord{a_i},\ord{k_i})$$
for some $k_i\geq 0$, we can combine them to from a morphism
$$h\in\pp(\ord{k},\ord{k_1+\cdots+k_n}).$$
Here $h$ is the unique morphism which for $1 \leq i \leq n$ makes each of the following diagrams commute:
\[
\xymatrix{
\ord{a_i} \ar[d]_{i_{f^{-1}(i)}} \ar[r]^{g_i}  & \ord{k_i} \ar[d] \\
\ord{k}  \ar[r]^-{h}                  & \ord{k_1+\cdots+ k_n}
}
\]
where the right hand vertical map is the order-preserving inclusion of the $i$-th block in the sum. We will use the notation
\begin{equation*}
h=f\langle g_1,\ldots,g_n\rangle.
\end{equation*}

Note that the morphism $f$ is only used to determine the blocks, so the definition of $f\langle g_1,\ldots,g_n\rangle$ makes sense also when $f\in\sets(\ord{k},\ord{n})$ and $g_i\in\p(\ord{a_i},\ord{k_i})$ and it produces an element of $\p(\ord{k},\ord{k_1+\cdots+k_n})$

It is not difficult to verify directly that we have the following identities of $\pp$-morphisms:
\begin{align}
\label{eq1}\tag{$\pp$1} i_I\circ f^I & =f\circ i_{f^{-1}(I)}\\
\label{eq2}\tag{$\pp$2} (fg)\langle h_1\circ g^{f^{-1}(1)}, \ldots, h_n\circ g^{f^{-1}(n)}\rangle & = f\langle h_1,\ldots,h_n\rangle\circ g\\
\label{eq3}\tag{$\pp$3} g^{f^{-1}(I)}\circ f^I &=(gf)^I\\
\label{eq4}\tag{$\pp$4} i_A\circ(i_B)^A & = i_{A\cap B}.
\end{align}
Note that all the above constructions and the formulas (\ref{eq1})-(\ref{eq4}) are preserved by the functors $\niceo\hookrightarrow\p\to\sets$.

%========================================================================
%========================================================================
\subsection{$\Delta$-sets, $\niceo$-sets and $\niceo\Sigma$-sets.}
\label{subsect:sets}

A $\pp$-set (where $\pp$ is $\Delta$, $\niceo$ or $\p$) is a contravariant functor from $\pp$ to the category of sets. We will write $X(n)$ instead of $X([n])$ or $X(\ord{n})$.  It is important to be clear that for a $\Delta$-set $X$ the set  $X(n)$ is to be thought of as a set of simplices \emph{of dimension $n$}, while for a $\niceo$-set or $\p$-set $X$ the set $X(n)$ is to be thought of as a set of simplices \emph{with $n$ vertices}. 
For an $\niceo$-set or $\p$-set $X$ we have the $i$-th face map $d_i=\delta_i^*:X(n)\to X(n-1)$ and $i$-th degeneracy $s_i=\sigma_i^*:X(n+1)\to X(n)$ for $1\leq i \leq n$. 

The inclusion $\Delta\hookrightarrow\niceo$ induces a forgetful functor $U:\sets^{\niceo^\op}\to\sets^{\Delta^\op}$ which satisfies $(UX)(n)=X(n+1)$ for $n\geq 0$ (it forgets the augmentation $X(0)$). For a simplicial set $Y$ we will denote by $Y_+$ the one-point augmentation, i.e. the $\niceo$-set with
\begin{displaymath}
Y_+(n)=\left\{\begin{array}{ll}
Y(n-1) & \textrm{ if } n\geq 1\\
\ast  & \textrm{ if } n=0.\end{array}\right.
\end{displaymath}
where $\ast$ is a singleton set.
% The structure maps of $Y_+$ are the shifts of the structure maps of $Y$ and the only structure map $Y_+(n)\to Y_+(0)$ is induced by the projection $Y(n-1)\to\pi_0 Y$. 
The functor $Y\mapsto Y_+$ is the right adjoint of $U$.

The inclusion $\niceo\hookrightarrow\p$ induces another forgetful functor $I:\sets^{\p^\op}\to\sets^{\niceo^\op}$. It has a left adjoint denoted $X\mapsto X\Sigma$. For an $\niceo$-set $X$ it is given by $(X\Sigma)(n)=X(n)\times\Sigma_n$ and the structure maps are defined by the formula
\begin{equation}
(x,\pi)\circ(g,\sigma) = (x\circ(\pi_*g),(g^*\pi)\circ\sigma)
\end{equation}
for $(x,\pi)\in X(n)\times\Sigma_n$ and $(g,\sigma)\in\p(\ord{m},\ord{n})$. We write the structure maps as acting on the right to indicate contravariance.

In particular, for $1\leq i\leq n$ the face map $d_i:(X\Sigma)(n)\to(X\Sigma)(n-1)$ is given by
\begin{equation}
\label{eqn:di}
d_i(x,\pi)=(x,\pi)\circ(\delta_i,\id_{n-1})=(x\circ\pi_*\delta_i,\delta_i^*\pi)=(x\delta_{\pi(i)},\delta_i^*\pi)=(d_{\pi(i)}x,d_i\pi)
\end{equation}
where $d_i\pi$ denotes the $(n-1)$-permutation obtained from $\pi$ by erasing the $i$-th position and reindexing.

For each $n\geq 0$ we have the canonical objects $\Delta_n$, $\niceo_n$ and $\p_n$ given by
$$\Delta_n(m)=\Delta([m],[n]),\newline
 \niceo_n(m)=\niceo(\ord{m},\ord{n}),\ \p_n(m)=\p(\ord{m},\ord{n})$$
and 
$${(\Delta_n)}_+=\niceo_{n+1}, \ U\niceo_n=\Delta_{n-1} \textrm{ for } n\geq 1,\ \niceo_n\Sigma=\p_n.$$

We can also consider categories of pairs $(X,X')$ where $X'$ is a sub-object of $X$. Note that since there are no morphisms in $\niceo$ or $\p$ with target $\ord{0}$, the $0$-component $X(0)$ of an $\niceo$-set or $\p$-set $X$ is a sub-object in a trivial way, so we can always form a pair $(X,X(0))$ of $\niceo$-sets or $\p$-sets. This gives a way to  remove the simplices with no vertices from either $\niceo$-sets or $\p$-sets.

For every $n\geq 0$ we have canonical pairs 
$$(\Delta_n,\partial\Delta_n),\ (\niceo_n,\partial\niceo_n),\ (\p_n,\partial\p_n)$$
in the respective categories. In each case the sub-object $\partial\pp_n$ consists of those morphisms whose underlying set map is not surjective.

The \emph{geometric realization} $|X|$ of an $\niceo$-set or $\p$-set $X$ is defined by passing to the underlying simplicial set, resp. $UX$ or $UIX$.

%========================================================================
%========================================================================
\subsection{Homotopical properties of $\p$-sets.}
\label{subsect:homotopy}

% The category $\p$ has a non-augumented version $\Delta\Sigma$ with objects $[n]$, $n\geq 0$, such that the obvious diagram of functors
% $$
% \xymatrix{
% \Delta \ar@{^{(}->}[r] \ar[d]& \Delta\Sigma \ar[d]\\
% \niceo  \ar@{^{(}->}[r] & \p
% }
% $$
% is commutative.

% Recall that we define the geometric realization of $\niceo$-sets and $\p$-sets by passing to the underlying simplicial set.

Let $X$ be a $\niceo$-set and let $\eta_X:X\to I(X\Sigma)$  be the unit of the adjunction $(-)\Sigma:\sets^{\niceo^\op}\rightleftarrows\sets^{\p^\op}:I$.  By \cite[Prop.5.1]{FL} there is a commutative diagram
$$
\xymatrix{
|X| \ar[r]^-{|\eta_X|} \ar[rd]_-{(x,\ast)} & |I(X\Sigma)| \ar[d]^-{\substack{(p_1,p_2)\\ \equiv}}\\
& |X|\times|\p_1|
}
$$
where $(p_1,p_2)$ is a homeomorphism and $|\p_1|$ is contractible by the argument of \cite[Ex.6]{FL}. It follows that $|\eta_X|:|X|\to|I(X\Sigma)|$ is always a homotopy equivalence. In particular, the spaces $|\p_n|$ are contractible for $n\geq 1$. 

In fact, one can say more. The category of $\Delta\Sigma$-sets, the obvious non-augmented version of $\niceo\Sigma$-sets, satisfies the assumptions of Theorem 6.2 of \cite{DHK} which provides it with a model structure in which a map is a weak equivalence if and only if it is a weak equivalence of the underlying simplicial sets. Then the adjoint pair $(-)\Sigma:\sets^{\Delta^\op}\rightleftarrows\sets^{\Delta\Sigma^\op}:I$ is a Quillen equivalence of model categories, in particular it induces an equivalence of homotopy categories.

% The homotopy theories of $\Delta$-sets and $\Delta\Sigma$-sets compare as follows.

% \begin{theorem}
% \begin{itemize}
% \item[a)] The category $\sets^{\Delta\Sigma^\op}$ is a model category. A map $f$ is a weak equivalence or fibration if the underlying map $Uf\in\sets^{\Delta^\op}$ is so and the cofibrations are the class $I$-cof for $I=\{\partial\Delta\Sigma_n\hookrightarrow\Delta\Sigma_n: n\geq 0\}$.
% \item[b)] The adjoint pair $(-)\Sigma:\sets^{\Delta^\op}\rightleftarrows\sets^{\Delta\Sigma^\op}:U$ is a Quillen equivalence. In particular $\mathrm{Ho}(\sets^{\Delta\Sigma^\op})\simeq \mathrm{Ho}(\sets^{\Delta^\op})$.
% \end{itemize}
% \end{theorem}
% \begin{proof}
% The category $\Delta\Sigma$ satisfies the assumptions of . By definition the forgetful functor $U$ preservers fibrations and trivial fibrations. Moreover if $A$ is an $\Delta\Sigma$-set, $Y$ is a $\Delta$-set and $f:A:\to Y\Sigma$ is a map of $\Delta\Sigma$-sets with adjoinf $f_\sharp: UA\to Y$ then the factorization $Uf=\eta_Y\circ f_\sharp$ shows that $f_\sharp$ is a weak equivalence if and only if $Uf$ is. It means that $((-)\Sigma,U)$ is a Quillen pair \cite[9.7]{DS} and the derived pair is a Quillen equivalence of homotopy categories.
% \end{proof}
%
% It follows that working with $\p$-sets we are homotopically in the same situation as working with $\niceo$-sets.

%========================================================================
%========================================================================
\section{The join operad of $\p$-sets.}
\label{section:3}
%========================================================================
%========================================================================
\subsection{Joins.}
The category of $\niceo$-sets is a natural context for the join operation at the level of simplicial sets.
%  It is described in \cite{Ehlers} in the context of simplicial sets augmented by simplices in degree $-1$, i.e. shifted with respect to our indexing. 
Let $X_1,\ldots, X_n$ be $\niceo$-sets. Then their join is the $\niceo$-set $\ojoin(X_1,\ldots,X_n)$ defined as
\begin{align}
\label{ojoin1}\ojoin(X_1,\ldots,X_n)(k)&=\coprod_{a_1+\cdots+a_n=k}X_1(a_1)\times\cdots\times X_n(a_n)\\
\label{ojoin2}&=\coprod_{\phi\in\niceo(\ord{k},\ord{n})}\ \prod_{i=1}^n X_i(\card{\phi^{-1}(i)}).
\end{align}
The two definitions are clearly equivalent because every order-preserving map $\phi\in\niceo(\ord{k},\ord{n})$ determines, and is determined, by an ordered partition of $k$ into $n$ parts of sizes $a_i=\card{\phi^{-1}(i)}$. The join has the following structure maps. An order-preserving map $f\in\niceo(\ord{k'},\ord{k})$ and a partition $a_1+\cdots+a_n=k$ of $k$ determine a new partition $a_1'+\cdots+a_n'=k'$ of $k'$ and a sequence of order-preserving maps $f_i:\ord{a_i'}\to\ord{a_i}$. Then for $(x_1,\ldots,x_n)\in X_1(a_1)\times\cdots\times X_n(a_n)$ we have
$$(x_1,\ldots,x_n)\circ f = (x_1f_1,\ldots,x_nf_n)\in X(a_1')\times \cdots\times X(a_n').$$
% {\em Is the notation used in this display explained properly in the previous section?}

There is an alternative formulation of this recipe which uses the definition (\ref{ojoin2}) of the join. If $f\in\niceo(\ord{k'},\ord{k})$ then $f$ takes the summand indexed by $\phi\in\niceo(\ord{k},\ord{n})$ to the summand of $\phi f\in\niceo(\ord{k'},\ord{n})$ as follows
$$(x_1,\ldots,x_n)\circ f = (x_1f^{\phi^{-1}(1)},\ldots,x_nf^{\phi^{-1}(n)})\in \prod_{i=1}^n X_i(\card{(\phi f)^{-1}(i)}).$$
Note that $f^{\phi^{-1}(i)}$ is exactly the morphism $f_i$ from the previous definition. It is straightforward to check using (\ref{eq3}) that these maps define the structure of an $\niceo$-set on the join.

Notice how this definition mimics the combinatorial structure of the join of two classical simplicial complexes. Indeed it follows from \cite{Ehlers} that if $Y_1,\ldots,Y_n$ are simplicial sets we have a homeomorphism
$$|\ojoin({Y_1}_+,\ldots,{Y_n}_+)|\equiv|Y_1|\ast\cdots\ast|Y_n|$$.

The category of $\p$-sets also has a join operation.  For $\p$-sets $X_1,\ldots, X_n$ we define
\begin{align}
\label{symjoin1}\symjoin(X_1,\ldots,X_n)(k)&=\coprod_{a_1+\cdots+a_n=k}X_1(a_1)\times\cdots\times X_n(a_n)\times_{\Sigma_{a_1}\times\cdots\times\Sigma_{a_n}}\Sigma_k\\
\label{symjoin2}  &=\coprod_{\phi\in\sets(\ord{k},\ord{n})}\ \prod_{i=1}^n X_i(\card{\phi^{-1}(i)}).
\end{align}

The equivalence of the two definitions follows from the fact that every map of sets $\phi\in\sets(\ord{k},\ord{n})$ can be factored as $\phi=f\pi$ with $\pi\in\Sigma_k$ and $f\in\niceo(\ord{k},\ord{n})$. In this factorization $f$ is determined uniquely and it induces a decomposition $a_1+\cdots+ a_n=k$ (with $a_i=\card{\phi^{-1}(i)}=\card{f^{-1}(i)}$), while $\pi$ is unique up to postcomposition with an element of $\Sigma_{a_1}\times\cdots\times\Sigma_{a_n}$. The maps
$$(x_1,\ldots,x_n)\to[(x_1(\pi^{-1})^{\phi^{-1}(1)}, \ldots, x_n(\pi^{-1})^{\phi^{-1}(n)}),\pi]$$
$$[(y_1,\ldots,y_n),\pi]\to(y_1\pi^{f^{-1}(1)}, \ldots, y_n\pi^{f^{-1}(n)})$$
establish the equivalence between the two versions of the join.
% \michal{Should I write some more explanation of this??}

The structure map of $\symjoin$ induced by the morphism $f\in\p(\ord{k'},\ord{k})$ takes the summand indexed by $\phi\in\sets(\ord{k},\ord{n})$ to the summand of $\phi f\in\sets(\ord{k'},\ord{n})$ by the formula
\begin{equation}
\label{eqn:joinstruct}
(x_1,\ldots,x_n)\circ f = (x_1f^{\phi^{-1}(1)},\ldots,x_nf^{\phi^{-1}(n)})\in \prod_{i=1}^n X_i(\card{(\phi f)^{-1}(i)}).
\end{equation}
The verification that these maps assemble to the structure of an $\p$-set is identical to that for $\ojoin$. 

A direct calculation with the representation (\ref{symjoin1}) of the join shows that if $f=(g,\sigma)\in\niceo(\ord{k'},\ord{k})\times\Sigma_{k'}$ then the structure map induced by $f$ is given by the formula
$$[(x_1,\ldots,x_n),\pi]\circ(g,\sigma) = [(x_1(\pi_*g)_1,\ldots,x_n(\pi_*g)_n),g^*\pi\circ\sigma]$$
where $(\pi_*g)_i:\ord{a_i'}\to\ord{a_i}$ are the components of the order-preserving map $\pi_* g:\ord{k'}\to\ord{k}$ determined by the partition $a_1+\cdots+a_n=k$. 

\begin{lemma} For any sequence of $\niceo$-sets $X_1,\ldots,X_n$ we have an isomorphism of $\p$-sets
\label{lem:symmetrization}
$$\symjoin(X_1\Sigma,\ldots,X_n\Sigma)=\ojoin(X_1,\ldots,X_n)\Sigma.$$
\end{lemma}
\begin{proof}
Since $X_i\Sigma(k)=X_i(k)\times\Sigma_k$ this follows immediately from (\ref{symjoin1}) and (\ref{ojoin1}).
\end{proof}

Because every map $h\in\niceo(\ord{k},\ord{k_1+\cdots+k_n})$ has a unique presentation in the form $h=f\langle g_1,\ldots,g_n\rangle$ where $f\in\niceo(\ord{k},\ord{n})$ and $g_i\in\niceo(\ord{\card{f^{-1}(i)}},\ord{k_i})$, we obtain isomorphisms
$$\Theta:\ojoin(\niceo_{k_1},\ldots,\niceo_{k_n})\xrightarrow{\simeq}\niceo_{k_1+\cdots+k_n}.$$
It follows from Lemma \ref{lem:symmetrization} that they yield isomorphisms
$$\Theta:\symjoin(\p_{k_1},\ldots,\p_{k_n})\xrightarrow{\simeq}\p_{k_1+\cdots+k_n}.$$
In both cases the map inducing the isomorphism acts on the summand of the join indexed by a map $\phi\in\sets(\ord{k},\ord{n})$ as
\begin{equation}
\label{eqn:standard-assoc}
\prod_{i=1}^n \pp(\ord{\card{\phi^{-1}(i)}},\ord{k_i})\ni(g_1,\ldots,g_n)\xrightarrow{\Theta}\phi\langle g_1,\ldots,g_n\rangle\in\pp(\ord{k},\ord{k_1+\cdots+k_n})
\end{equation}
where $\pp$ is $\niceo$ or $\p$. We will frequently use the formal similarity between (\ref{ojoin2}) and (\ref{symjoin2}) to present a single argument that simultaneously applies to analogous statements about $\ojoin$ and $\symjoin$.

A slight generalization of this example shows that the join operations are associative, that is $J(X,J(Y,Z))=J(J(X,Y),Z)=J(X,Y,Z)$. The generalization to more than three sets is obvious, but we state it in the next lemma to introduce the notation which will be used later.

\begin{lemma}
\label{lemma:associative}
Let $X_{1,*},\ldots,X_{n,*}$ be sequences of $\niceo$-sets or $\p$-sets, where the $i$-th sequence has length $k_i$ for $i=1,\ldots,n$. Let $X_{*,*}$ denote the sequence of length $\sum k_i$ obtained by joining the given sequences in the lexicographic order of indices. Then there is an isomorphism
$$\Theta: J(J(X_{1,*}),\ldots,J(X_{n,*}))\to J(X_{*,*})$$
where $J=\ojoin$ or $J=\symjoin$.
\end{lemma}

The easy proof of this lemma is postponed until Section \ref{sect:appendix}. From now we will often simplify notation in this way by writing $X_*$ instead of $(X_1,\ldots,X_n)$.

It follows that the join $\ojoin(X,Y)$ endows the category of $\niceo$-sets with a monoidal structure with unit $\niceo_0$. This monoidal operation, however, is not symmetric. Indeed, the $\niceo$-sets $\ojoin(X,Y)$ and $\ojoin(Y,X)$ have isomorphic sets of simplices but there is no isomorphism between these sets which commutes with the structure maps.  Intuitively, they have the same simplices but the vertices of these simplices are ordered differently. 

This deficiency is corrected by the join $\symjoin$ of $\p$-sets. The category of $\p$-sets becomes a \emph{symmetric} monoidal category with join $\symjoin(X,Y)$, unit $\p_0$ and the symmetry isomorphism described in the next lemma.

\begin{lemma}
\label{lem:equiv-join}
For any sequence $X_1,\ldots,X_n$ of $\p$-sets and a permutation $\sigma\in\Sigma_n$ there is a natural isomorphism of $\p$-sets
$$T_\sigma:\symjoin(X_1,\ldots,X_n)\to\symjoin(X_{\sigma^{-1}(1)},\ldots,X_{\sigma^{-1}(n)}).$$
\end{lemma}
\begin{proof}
We describe the map $T_\sigma$ using the presentation (\ref{symjoin2}). In degree $k$ the map $T_\sigma$ sends the summand of $\symjoin(X_*)$ indexed by $\phi\in\sets(\ord{k},\ord{n})$ to the summand of $\symjoin(X_{\sigma^{-1}(*)})$ indexed by $\sigma\phi\in\sets(\ord{k},\ord{n})$ shuffling the factors appropriately. This left action commutes with the right actions of the structure morphisms, hence $T_\sigma$ becomes a map of $\p$-sets. It is immediate that $T_{\pi\sigma}=T_\pi\circ T_\sigma$.
\end{proof}
We can transcribe this description of the isomorphism $T_\sigma$ to the presentation (\ref{symjoin1}). Given $\sigma\in\Sigma_n$ and $a_1+\cdots+a_n=k$ let $\sigma_{a_1,\ldots,a_n}\in\Sigma_k$ be the block permutation determined by $\sigma$ which permutes the blocks of sizes $a_1,\ldots,a_n$ in the way $\sigma$ permutes $n$ letters. Formally
$\sigma_{a_1,\ldots,a_n} = \sigma_* f$ where $f\in\niceo(\ord{k},\ord{n})$ is the order-preserving map corresponding to the partition $a_1+\cdots+a_n=k$. Then
$$T_\sigma([(x_1,\ldots,x_n),\pi])=[(x_{\sigma^{-1}(1)},\ldots,x_{\sigma^{-1}(n)}),\sigma_{a_1,\ldots,a_n}\pi].$$
In particular, the basic symmetry operator of the monoidal structure 
$$T_{X,Y}:\symjoin(X,Y)\to\symjoin(Y,X)$$
acts by sending the element $[(x,y),\pi]\in X(a)\times Y(b)\times_{\Sigma_a\times\Sigma_b}\Sigma_{a+b}$ to $$[(y,x),\tau_{a,b}\pi]\in Y(b)\times X(a)\times_{\Sigma_b\times\Sigma_a}\Sigma_{b+a}$$ where $\tau_{a,b}$ switches the two blocks of sizes $a$ and $b$.

%========================================================================
%========================================================================
\subsection{Canonical maps and the join operad of $\p$-sets.}
We can now define the $\p$-analogue of the canonical maps of the  introduction. For a sequence of $\p$-sets $X_1,\ldots,X_n$ let $A(X_1,\ldots,X_n)$, also denoted $A(X_*)$, be the $\p$-set
\begin{equation*}
A(X_*)=\p_n\times X_1\times\cdots\times X_n.
\end{equation*}
We define a natural map of $\p$-sets
\begin{equation*}
\alpha: A(X_*)\to\symjoin(X_*)
\end{equation*}
as follows. Write an element of $A(X_*)(k)=\p(\ord{k},\ord{n})\times X_1(k)\times\cdots\times X_n(k)$ as $(f;x_1,\ldots,x_n)$. Then $\alpha(f;x_1,\ldots,x_n)$ lies in the summand of $\symjoin(X_*)$ indexed by the underlying map $f\in\sets(\ord{k},\ord{n})$ and
\begin{equation}
\label{eqn:alpha}
\alpha(f;x_1,\ldots,x_n)=(x_1i_{f^{-1}(1)},\ldots,x_ni_{f^{-1}(n)}).
\end{equation}
If $f=(g,\sigma)$ then we can also write
\begin{equation}
\label{eqn:alpha2}
\alpha((g,\sigma);x_1,\ldots,x_n)=[(x_1(\sigma^{-1}\circ i_{g^{-1}(1)}),\ldots,x_n(\sigma^{-1}\circ i_{g^{-1}(n)})),\sigma].
\end{equation}
Again, it is not difficult to check that $\alpha$ is a well defined map of $\p$-sets. For every $n$ it is, in fact, a natural transformation between the functors $A,\symjoin:(\sets^{\p^\op})^{\times n}\to\sets^{\p^\op}$. The precise argument is given in Section \ref{sect:appendix}.
Moreover it is immediate to check that this transformation is $\Sigma_n$-equivariant in the sense that
\begin{equation}
\label{eq:alphaequivariant}
T_\pi(\alpha(f;x_1,\ldots,x_n)) = \alpha(\pi f;x_{\pi^{-1}(1)},\ldots,x_{\pi^{-1}(n)}).
\end{equation}

If we specialize to the case where $X_i=\p_{k_i}$, then $\alpha$, together with the associativity isomorphisms of (\ref{eqn:standard-assoc}), gives maps
\begin{equation}
\label{map:phi}
\Psi_{k_1,\ldots,k_n}: \p_n\times \p_{k_1}\times\cdots\times \p_{k_n}\to\symjoin(\p_{k_1},\ldots,\p_{k_n})=\p_{k_1+\cdots+k_n}
\end{equation}
which, by (\ref{eqn:alpha}) and (\ref{eqn:standard-assoc}), are given by the formula
\begin{equation}
\label{map:phi2}
\Psi_{k_1,\ldots,k_n}(f;g_1,\ldots,g_n) = f\langle g_1\circ i_{f^{-1}(1)},\ldots, g_n\circ i_{f^{-1}(n)}\rangle \ \in \p(\ord{k},\ord{k_1+\cdots+k_n})
\end{equation}
for $(f;g_1,\ldots,g_n)\in\p(\ord{k},\ord{n})\times\p(\ord{k},\ord{k_1})\times\cdots\times\p(\ord{k},\ord{k_n})$.

To discuss operads and their coactions we need the following lemma.

\begin{proposition}
\label{prop:atojoperadic}
Let $X_{1,*},\ldots,X_{n,*}$ be sequences of $\p$-sets as in Lemma \ref{lemma:associative}. Then the following diagram commutes
\begin{equation*}
\xymatrix{
A(A(X_{1,*}),\ldots,A(X_{n,*})) \ar[r]^{A(\alpha,\ldots,\alpha)}\ar[dd]_{\Psi_{k_1,\ldots,k_n}} & A(\symjoin(X_{1,*}),\ldots,\symjoin(X_{n,*})) \ar[d]^\alpha\\
 & \symjoin(\symjoin(X_{1,*}),\ldots,\symjoin(X_{n,*}))\ar[d]^{\Theta}\\
A(X_{*,*})\ar[r]_\alpha & \symjoin(X_{*,*})
}
\end{equation*}
\end{proposition}
The proof of this result is postponed until the last section.

We can now define the symmetric join operad in the category of $\p$-sets. 

\begin{theorem}
\label{prop:psetoperad}
The sequence of $\p$-sets $\{\p_n\}_{n\geq 0}$ forms an operad (non-unital, with permutations) in the symmetric monoidal category of $\p$-sets $(\sets^{\p^\op},\times,\ast)$. The structure maps of the operad are the $\Psi_{k_1,\ldots,k_n}$ of (\ref{map:phi}). The right action of $\Sigma_n$ on $\p_n$ is given by
$$f\circ\pi =\pi^{-1}f$$
for $f\in \p_n(k)$ and $\pi\in\Sigma_n$.
\end{theorem}
Again, the proof can be found in Section \ref{sect:appendix}. The most tedious part is the verification of the operadic associativity axiom. The operad has no unit since there is no map of $\p$-sets $\ast\to\p_1$.

Until now we have constructed an operad of $\p$-sets which governs the relations between the canonical maps of the symmetric join $\symjoin$. In addition to the operad structure the transformation $\alpha$, precomposed with the $n$-fold diagonal $X\to X\times \cdots\times X$, yields for every $\p$-set $X$ and every $n\geq 0$ maps of $\p$-sets
\begin{equation}
\label{osigma-coaction-nonrel}
\Psi_n^X:\p_n\times X \to \symjoin(X^n)
\end{equation}
where $\symjoin(X^n)=\symjoin(X,\ldots,X)$.

%===============================================================================
%===============================================================================
\subsection{The induced operads of pairs.} 

We now explain how the operad $\{\p_n\}_{n\geq 0}$ gives rise to an operad in the category of pairs of $\p$-sets and, via the forgetful functor $I$, an operad in the category of pairs of $\niceo$-sets. 

First, given a sequence $X_1,\ldots,X_n$ of $\niceo$-sets or $\p$-sets we define the \emph{relative join} as the pair
$$(J(X_1,\ldots,X_n),\partial J(X_1,\ldots,X_n))$$
for $J=\ojoin$ or $J=\symjoin$ where $\partial J(X_1,\ldots,X_n)$ is the subobject consisting of the summands in (\ref{ojoin2}) or (\ref{symjoin2}) indexed by non-surjective maps $\phi$. In other words those are the simplices of the join which do not contain a proper face from at least one of the factors.

The maps $\Psi_{k_1,\ldots,k_n}$ of (\ref{map:phi}) clearly induce maps of relative $\p$-sets
\begin{eqnarray}
\label{partialmaps}
\Psi_{k_1,\ldots,k_n}: & (\p_n,\partial \p_n)\times (\p_{k_1},\partial \p_{k_1})\times\cdots\times(\p_{k_n},\partial \p_{k_n})\to\\
\nonumber &\to(\p_{k_1+\cdots+k_n},\partial\p_{k_1+\cdots+k_n})
\end{eqnarray}
and it is an immediate corollary of Theorem \ref{prop:psetoperad} that $\{(\p_n,\partial \p_n)\}_{n\geq 0}$ is an operad in the category of pairs. Moreover, in the relative context the map $\alpha$ induces a map
\begin{equation}
\label{relativealpha}
\alpha:(\p_n,\partial\p_n)\times (X_1,X_1(0))\times\cdots\times(X_n,X_n(0))\to(\symjoin(X_*),\partial\symjoin(X_*))
\end{equation}
which specializes to a relative version of (\ref{osigma-coaction-nonrel})
\begin{equation}
\label{osigma-coaction}
\Psi_n^X:(\p_n,\partial\p_n)\times (X,X(0))\to(\symjoin(X^n),\partial\symjoin(X^n)).
\end{equation}

\section{The chain operad}
\label{section:4}

In this section we explain how to construct an $E_\infty$-operad $\frakj$ in the category of chain complexes from the symmetric join operad.  We go on to explain how the maps $\Psi_n^X$ of (\ref{osigma-coaction}) give a coaction of the operad $\frakj$ on the chain complex $C_*(Y)$ of a simplicial set $Y$. 

%===============================================================================
%===============================================================================
\subsection{Conventions regarding chain complexes.} 
We fix once for all some commutative ground ring $k$. Let $(\ch,\otimes,k)$ denote the symmetric monoidal category of chain complexes of $k$-modules, with differential of degree $-1$. The symmetry operator is $$T(x\otimes y)=(-1)^{\deg(x)\deg(y)}y\otimes x.$$

For a chain complex $(C,d)$ let $\chsusp^n C$ denote the $n$-fold chain suspension, i.e. the chain complex with $(\chsusp^n C)_m=C_{m-n}$ and differential $d(\chsusp^n x)=(-1)^n\chsusp^ndx$. For any chain complexes $C$ and $D$ we have an isomorphism
$$\chsusp^n C\otimes \chsusp^m D=\chsusp^{n+m}(C\otimes D)$$
given by $\chsusp^nx\otimes\chsusp^m y\to(-1)^{m\deg(x)}\chsusp^{n+m}(x\otimes y)$.

If $Y$ is a simplicial set then $C_*(Y)$ denotes its usual chain complex.
%  is defined by $C_n(Y)=k[Y(n)]$ with differential $dy=\sum_{i=0}^n(-1)^id_iy$ for $y\in Y(n)$. The normalized chain complex $N_*(Y)$ is defined in the same way using only non-degenerate simplices.
The functor
$$C_*(-):(\sets^{\Delta^\op},\times,\ast)\to(\ch,\otimes,k)$$
is lax-monoidal with the natural transformation
$$C_*(Y_1)\otimes C_*(Y_2)\xrightarrow{EZ} C_*(Y_1\times Y_2)$$
given by the Eilenberg-Zilber map \cite[Def.29.7]{May}.

%===============================================================================
%===============================================================================
\subsection{Chain complexes of $\niceo$-sets.}
If $X$ is an $\niceo$-set we define the augmented chain complex $C_*^\niceo(X)$ by
$$C_n^\niceo (X)=k[X(n)]$$
with differential $dx=\sum_{i=1}^n(-1)^id_ix$ for $x\in X(n)$. If $(X,X')$ is an $\niceo$-set pair then we can form the relative chain complex $C_*^\niceo(X,X')$ in the usual way.

Recall the adjoint functors $U:\sets^{\niceo^\op}\rightleftarrows\sets^{\Delta^\op}:(-)_+$ between $\niceo$-sets and $\Delta$-sets. The sign conventions associated to suspensions imply that
\begin{eqnarray}
\label{eqn:reduced}
C_*^\niceo(X,X(0))&=&\chsusp C_*(UX)\quad \textrm{ for any } \niceo\textrm{-set } X\\
\nonumber C_*^\niceo(Y_+,Y_+(0))&=&\chsusp C_*(Y)\quad\textrm{ for any } \Delta\textrm{-set } Y
\end{eqnarray}

\textbf{Remark.} The last isomorphism holds for \emph{any} functorial augmentation $Y_+$ of simplicial sets. In fact we will see that the choice of augmentation will not affect the final outcome of the constructions of this section. This is not surprising, since we are trying to produce chain level maps for $\Delta$-sets, while $\niceo$-sets and $\p$-sets serve only as an intermediate tool.

We also have
\begin{eqnarray}
\label{eqn:ojoinmonoidal}
C_*^\niceo(\ojoin(X_1,\ldots,X_n))&=&\bigotimes_{i=1}^n C_*^\niceo(X_i)\\
\nonumber C_*^\niceo(\ojoin(X_1,\ldots,X_n),\partial\ojoin(X_1,\ldots,X_n))&=&\bigotimes_{i=1}^n C_*^\niceo(X_i,X_i(0))
\end{eqnarray}
for any sequence $X_1,\ldots,X_n$ of $\niceo$-sets. In other words, the functor $C_*(-):(\sets^{\niceo^\op},\ojoin,\niceo_0)\to(\ch,\otimes,k)$ is monoidal. On the other hand, using the isomorphisms (\ref{eqn:reduced}), we see that the Eilenberg-Zilber map induces in the augmented context a transformation
\begin{equation}
\label{eqn:ez-monoidal}
\chsusp^{-1}C_*^\niceo(X_1,X_1(0))\otimes\chsusp^{-1}C_*^\niceo(X_2,X_2(0))\xrightarrow{EZ}\chsusp^{-1}C_*^\niceo(X_1\times X_2,X_1(0)\times X_2(0)).
\end{equation}
which makes the functor $\chsusp^{-1}C_*^\niceo(-,-(0)):(\sets^{\niceo^\op},\times,\ast)\to(\ch,\otimes,k)$ lax-monoidal.
% The same remarks apply to the normalized chain complexes $N_*^\niceo(X)$.

%===============================================================================
%===============================================================================
\subsection{Chain complexes of $\p$-sets.}
If $Z$ is an $\p$-set we will continue to write $C_*^\niceo(Z)$ for the chain complex of the underlying $\niceo$-set $IZ$. Recall from \ref{subsect:homotopy} that for any $\niceo$-set $X$ the standard inclusion $\eta_X:X\to I(X\Sigma)$ is a weak equivalence. There is no natural inverse map $X\Sigma\to X$ but a suitable inverse exists at the level of chain complexes.

\begin{proposition}
\label{prop:strange}
For every $\niceo$-set $X$ the assignment 
$$(x,\pi)\to\sgn(\pi)x$$
induces a natural map of chain complexes
\begin{equation}
s_X:C_*^\niceo(X\Sigma)\to C_*^\niceo(X)
\end{equation}
which gives a quasi-isomorphism
$C_*^\niceo(X\Sigma,X\Sigma(0))\xrightarrow{\sim} C_*^\niceo(X,X(0))$.
\end{proposition}
\begin{proof}

Let us first verify that $s=s_X$ is indeed a map of chain complexes. We have
\begin{align*}
ds(x,\pi)=& \sgn(\pi)dx=\sgn(\pi)\sum_{i=1}^n(-1)^id_ix.
\end{align*}
Using (\ref{eqn:di}) and the easy formula
\begin{equation}
\label{eqn:signs}
\sgn(d_i\pi)=(-1)^{i+\pi(i)}\sgn(\pi)
\end{equation}
we verify that
\begin{align*}
sd(x,\pi)=&\ s\sum_{i=1}^n(-1)^i(d_{\pi(i)}x,d_i\pi)=\\
=&\ \sum_{i=1}^n(-1)^i\sgn(d_i\pi)d_{\pi(i)}x=\\
=&\ \sgn(\pi)\sum_{i=1}^n(-1)^{\pi(i)}d_{\pi(i)}x=\ \sgn(\pi)\sum_{i=1}^n(-1)^{i}d_{i}x=ds(x,\pi)
\end{align*}
so the claim is proved. 

The map $s_X$ is clearly natural. To prove that the map of relative complexes is a quasi-isomorphism note that the identity of $C_*^\niceo(X,X(0))$ factors as
$$C_*^\niceo(X,X(0))\xrightarrow{C_*^\niceo(\eta_X)}C_*^\niceo(X\Sigma,X\Sigma(0))\xrightarrow{s_X}C_*^\niceo(X,X(0)),$$
and $C_*^\niceo(\eta_X)$ is a quasi-isomorphism by the results of Section \ref{subsect:homotopy}.

% Let $|\cdot|$ denote the geometric realization of a simplicial set and let $S(\cdot)$ be the augumented singular simplicial set of a topological space. According to \cite[Prop.5.1.iii]{FL} the identity of $|X|$ factors as
%$$|X|\xrightarrow{|\eta_X|}|X\Sigma|\xrightarrow{\ \cong\ }|X|\times|\p_1|\rightarrow |X|$$
%where the middle map is a homeomorphism and the last projection is a homotopy equivalence since $|\p_1|$ is contractible. It follows that the map $|\eta_X|$ is a homotopy equivalence. In the commutative diagram
% \begin{equation*}
% \xymatrix{
% X \ar[r]^{\eta_X}\ar[d] &  \ar[d] X\Sigma\\
% S(|X|) \ar[r]_{S(|\eta_X|)} & S(|X\Sigma|)
% }
% \end{equation*}
% the bottom map and the canonical vertical maps induce isomorphisms in homology (\cite[Prop.16.2]{May}), hence so does $\eta_X$.
% \michal{Move it to a different section on homotopy categories.}
\end{proof}

Consider now the chain morphism $s_X$ in the special case when $X=\ojoin(X_1,\ldots,X_n)$. Then by Lemma \ref{lem:symmetrization} and (\ref{eqn:ojoinmonoidal}) $s_X$ can be identified with the map
$$C_*^\niceo(\symjoin(X_*\Sigma))=C_*^\niceo(\ojoin(X_*)\Sigma)\xrightarrow{s_{X}}C_*^\niceo(\ojoin(X_*))=\bigotimes_{i=1}^nC_*^\niceo(X_i).$$
% which makes it a map between two $\Sigma_n$-chain complexes.
\begin{proposition}
\label{strange:equivariance}
For any $\niceo$-sets $X_1,\ldots,X_n$ the maps $$s_{\ojoin(X_*)}:C_*^\niceo(\symjoin(X_1\Sigma,\ldots,X_n\Sigma))\to\bigotimes_{i=1}^nC_*^\niceo(X_i)$$ 
are $\Sigma_n$-equivariant.% with respect to the symmetry isomorphisms of the monoidal structures $(\sets^{\p^\op}, \symjoin)$ and $(\ch,\otimes)$.
\end{proposition}
\begin{proof}
Due to obvious associativity it suffices to check the claim for the join of two objects, that is to verify the commutativity of the diagram
\begin{equation*}
\xymatrix{
C_*^\niceo(\symjoin(X\Sigma,Y\Sigma)) \ar[r]^{s_{\ojoin(X,Y)}} \ar[d]_{C_*^\niceo(T_{X\Sigma,Y\Sigma})} & C_*^\niceo(X)\otimes C_*^\niceo(Y)\ar[d]^{T_{C_*^\niceo(X),C_*^\niceo(Y)}}\\
C_*^\niceo(\symjoin(Y\Sigma,X\Sigma))\ar[r]^{s_{\ojoin(Y,X)}} & C_*^\niceo(Y)\otimes C_*^\niceo(X).\\
}
\end{equation*}
If an element of $\symjoin(X\Sigma,Y\Sigma)(k)=\ojoin(X,Y)\Sigma(k)$ is represented by the triple $(x_p,y_q,\pi)$ with $p+q=k$ then 
$$T s(x,y,\pi)=\sgn(\pi)T(x\otimes y)=(-1)^{pq}\sgn(\pi)y\otimes x$$
while (see the discussion following Lemma \ref{lem:equiv-join}) 
$$s T(x,y,\pi)=s(y,x,\tau_{p,q}\pi)=\sgn(\tau_{p,q})\sgn(\pi)y\otimes x$$
and the two values are equal because $\sgn(\tau_{p,q})=(-1)^{pq}$.
\end{proof}

%=============================================================
%=============================================================
\subsection{Construction of the chain operad.}
\label{subsect:constr}
The structure maps of the operad $\{(\p_n,\partial\p_n)\}_{n\geq 0}$ of pairs of $\p$-sets (\ref{partialmaps}) induce, via the lax-monoidal functor $\chsusp^{-1}C_*^\niceo(-)$ of (\ref{eqn:ez-monoidal}), maps of chain complexes
\begin{eqnarray}
\fraka_{k_1,\ldots,k_n}:&\chsusp^{-1}C_*^\niceo(\p_n,\partial\p_n)\otimes\bigotimes_{i=1}^n\chsusp^{-1}C_*^\niceo(\p_{k_i},\partial\p_{k_i})\to \\
\nonumber &\to\chsusp^{-1}C_*^\niceo(\p_{k_1+\cdots+k_n},\partial\p_{k_1+\cdots+k_n})
\end{eqnarray}
which are therefore the structure maps of a chain operad
% \michal{Check $n=0$.}
$$\fraka(n)=\chsusp^{-1}C_*^\niceo(\p_n,\partial\p_n),\qquad {n\geq 0}.$$
In particular $\fraka(0)=\chsusp^{-1}k$. It is also easy to check that if $\id\in\p(\ord{1},\ord{1})$ denotes the unique morphism then $\chsusp^{-1}\id\in\fraka(1)_0$ is the unit of this operad. 

Recall that for any $\p$-set we have the maps of (\ref{osigma-coaction}). Now let $X$ be an $\niceo$-set. Consider the following composition $\fraka_n^X$.

\begin{eqnarray*}
\fraka^X_n:\fraka(n)\otimes C_*^\niceo(X,X(0))&=&\ \chsusp^{-1}C_*^\niceo(\p_n,\partial\p_n)\otimes C_*^\niceo(X,X(0))\\
&\myrightarrow{1\otimes{\eta_X}_*}&\ \chsusp^{-1}C_*^\niceo(\p_n,\partial\p_n)\otimes C_*^\niceo(X\Sigma,X\Sigma(0))\\
&\myrightarrow{EZ}&\ C_*^\niceo((\p_n,\partial\p_n)\times (X\Sigma,X\Sigma(0)))\\
&\myrightarrow{{\Psi^{X\Sigma}_n}_*}&\ C_*^\niceo(\symjoin(X\Sigma^n),\partial\symjoin(X\Sigma^n))\\
&=&\ C_*^\niceo(\ojoin(X^n)\Sigma,\partial\ojoin(X^n)\Sigma)\\
&\myrightarrow{s_{\ojoin(X^n)}}&\ C_*^\niceo(\ojoin(X^n),\partial\ojoin(X^n))\\
&=&\ C_*^\niceo(X,X(0))^{\otimes n}
\end{eqnarray*}

In summary, we first enlarge $X$ to the $\p$-set $X\Sigma$, apply the simplicial `coaction' maps of (\ref{osigma-coaction}) and pass back to the non-symmetric context using Proposition \ref{prop:strange}.

\begin{proposition} 
\label{pre-operad}
For any $\niceo$-set $X$ the maps
$$\fraka^X_n:\fraka(n)\otimes C_*^\niceo(X,X(0))\to C_*^\niceo(X,X(0))^{\otimes n}$$
equip $C_*^\niceo(X,X(0))$ with a natural structure of a coalgebra over the operad $\{\fraka(n)\}_{n\geq 0}$.
\end{proposition}
For the proof see the last section.
% $$
% \xymatrix{
% \fraka(n)\otimes\fraka(k_1)\otimes\cdots\otimes\fraka(k_n)\otimes C \ar[rr]\ar[d] & & \fraka(k_1)\otimes\cdots\otimes\fraka(k_n)\otimes C^{\otimes n}\ar[d]\\
% \fraka(k_1+\ldots+k_n)\otimes C \ar[rd]& & \fraka(k_1)\otimes C\otimes\cdots\otimes\fraka(k_n)\otimes C \ar[ld]\\
% & C^{\otimes k_1+\ldots+k_n} & 
% }
% $$
% where $C=C_*^\niceo(X,X(0))$ follows by applying the functor $C_*^\niceo(-,-(0))$ to the diagram of Proposition \ref{prop:yetanotherdiagram} precomposed with the diagonal map of $X$. That ends the verification.

Now suppose that $X=Y_+$ for a simplicial set $Y$. Then $C_*^\niceo(Y_+,Y_+(0))=\chsusp C_*(Y)$ and we conclude that $\chsusp C_*(Y)$ is a coalgebra over $\fraka$. Now we use the device known as operadic desuspension\michal{or suspension?} of $\fraka$ i.e. the operad $\Lambda\fraka$ characterized by the property that giving $\chsusp C$ the structure of a coalgebra over $\fraka$ is the same as giving $C$ the structure of a coalgebra over $\Lambda\fraka$. Explicitly
\begin{equation*}
(\Lambda\fraka)(n)=\chsusp^{1-n}\fraka(n)\otimes\sgn_n
\end{equation*}
where $\sgn_n$ is the sign representation of $\Sigma_n$ (see \cite{Jones}).

\begin{definition}
\label{defn:operad}
The \emph{symmetric join operad} $\{\frakj(n)\}_{n\geq 0}$ is the operad $\frakj=\Lambda\fraka$ in the category of chain complexes. Explicitly
$$\frakj(n)=\chsusp^{-n}C_*^\niceo(\p_n,\partial\p_n)\otimes\sgn_n.$$
\end{definition}

\begin{theorem}The symmetric join operad $\{\frakj(n)\}_{n\geq 0}$ is a unital $E_\infty$-operad of chain complexes. For any simplicial set $Y$ the chain complex $C_*(Y)$ is naturally a $\frakj$-coalgebra (hence $C^*(Y)$ is a $\frakj$-algebra).
\end{theorem}
\begin{proof}
Each $\frakj(n)_d$ is clearly a free $k[\Sigma_n]$-module. Since $C_d^\niceo(\p_n,\partial\p_n)=0$ for $d<n$, each chain complex $\frakj(n)$ is concentrated in non-negative degrees. Because $(\p_n,\partial\p_n)=(\niceo_n,\partial\niceo_n)\Sigma$, Proposition \ref{prop:strange} provides for $n\geq 1$ a quasi-isomorphism
$$C_*^\niceo(\p_n,\partial\p_n)\to C_*^\niceo(\niceo_n,\partial\niceo_n)=\chsusp C_*({\Delta_{n-1}},\partial{\Delta_{n-1}})$$
where the last complex has one-dimensional homology group concentrated in degree $n$. Moreover $\frakj(0)=k$. The other statements follow from the properties of the operad $\fraka$.

%The second statement is a consequence of the previous discussion and the definition of the operadic desuspension. It is clear that every $\frakj(n)$ is a free $k[\Sigma_n]$-module since the action of $\Sigma_n$ on the set of surjections $f:k\to n$ is free. The unit of the join operad is the element $\chsusp^{-1}\id\in\chsusp^{-1}\p(1,1)=\frakj(1)_0$. Indeed, it suffices to see that the same element $\chsusp^{-1}\id\in\fraka(1)_0$ acts on each $C_*^\niceo(X,X(0))$ by identity. This holds because in the composition
%$$\fraka(1)\otimes C_*^\niceo(X,X(0))=\chsusp(\chsusp^{-1}C_*^\niceo(\p_1,\partial\p_1)\otimes\chsusp^{-1} C_*^\niceo(X,X(0)))\xrightarrow{EZ} C_*^\niceo((\p_1,\partial\p_1)\times (X,X(0)))\to C_*^\niceo(X,X(0))$$
%the E.-Z. map sends $\chsusp^{-1}\id\otimes\chsusp^{-1}x$ to $(\id,x)$ and the second map is induced by the map $\p_1\times X\to X$ of the operad of $\p$-sets, which is just a projection.

%Finally,  The obesrvation that $C_*^\niceo(\p_n,\partial\p_n)$ is zero in degrees less than $n$ completes the proof.
\end{proof}

Let us make a few remarks.

\begin{itemize}
\item 
Since a surjective morphism in $\p(\ord{n+1},\ord{n})$ must send two elements of $\ord{n+1}$ to the same value and be injective otherwise, one can observe that $\frakj(n)_0=k[\Sigma_n]$ splits as a direct sum
$$\frakj(n)_0 = k[\id_n-\sgn(\pi)\pi]_{\pi\in\Sigma_n}\ \oplus\ k[\id_n]$$
where additionally the first summand is precisely the image of the differential $d:\frakj(n)_1\to\frakj(n)_0$. It follows that the sign map 
$$\sgn:\frakj(n)_0\to k$$
is an augmentation of $\frakj(n)$ and it defines a quasi-isomorphism $\frakj\to\mathfrak{Com}$ to the commutative operad.

\item 
The degree zero maps 
$$\frakj(n)_0\otimes\frakj(k_1)_0\otimes\cdots\otimes\frakj(k_n)_0\to\frakj(k_1+\cdots+ k_n)_0$$ agree, up to sign, with the maps
$$k[\Sigma_n]\otimes k[\Sigma_{k_1}]\otimes\cdots\otimes  k[\Sigma_{k_n}]\to k[\Sigma_{k_1+\cdots+ k_n}]$$
induced by the canonical permutation operad $\{\Sigma_n\}_{n\geq 0}$ in the category of sets \cite[Sec. 0.10]{BerFres04}.

\removed{
\item 
Instead of the chain functor $C_*^\niceo$ we could use the normalized chain functor $N_*^\niceo$ throughout the entire construction. This way we obtain the normalized symmetric join operad
$$\frakj^N(n)=\chsusp^{-n}N^\niceo_*(\p_n,\partial\p_n)\otimes\sgn_n$$
which is an $E_\infty$-operad coacting naturally on the normalized chains $N_*(Y)$ of simplicial sets $Y$. Note that a morphism $f\in\p_n(k)=\p(\ord{k},\ord{n})$ is non-degenerate as a simplex of the underlying $\niceo$-set if and only if there is no $i\in \ord{k}$ such that $f(i)=f(i+1)$ and $i$ immediately precedes $i+1$ in the ordering of their common fibre of $f$.
}

\item 
Since the abelian groups $\frakj(n)_d$ are generated by $\p$-morphisms whose underlying set maps are surjections it is quite natural to expect a relation between the join operad and the sequence operad $\nices$ of \cite{McSmith03}. Indeed, the join operad maps to the sequence operad by a map induced, up to sign, by the forgetful functor $\p\to\sets$. More effort goes into getting the correct signs.
% , but as a reward we obtain another natural explanation of the sign conventions used in the sequence operad.

Let us remind after \cite{McSmith03,BerFres04} that we call a map $f\in\sets(\ord{m},\ord{n})$ \emph{nondegenerate} if it is a surjection and $f(i)\neq f(i+1)$ for all $1\leq i<m$. The sequence operad is an operad $\{\nices(n)\}_{n\geq 0}$ of chain complexes such that $\nices(n)_d$ is a free $k$-module generated by all nondegenerate surjections $f:\ord{n+d}\to\ord{n}$. The differential, symmetric group action and operadic composition are given in 2.18, 2.19 and 2.26 of \cite{McSmith03}.

Any map $f\in\sets(\ord{m},\ord{n})$ has a unique decomposition as $f=g_f\pi_f$ where $g_f\in\niceo(\ord{m},\ord{n})$, $\pi_f\in\Sigma_{m}$ and $\pi_f$ is order-preserving on each fibre $f^{-1}(i)$. Define the sign of $f$ as
\begin{equation*}
s(f)=(-1)^{mn}\sgn(\pi_f)(-1)^{\sum_{j=1}^m f(j)}.
\end{equation*}
The map $S:\frakj(n)_{m-n}\to\nices(n)_{m-n}$ is defined by
\begin{equation*}
S(\chsusp^{-n}f)=s(f)\cdot f
\end{equation*}
for a surjective morphism $f\in\p(\ord{m},\ord{n})$ where the result is $0$ if the underlying set map of $f$ is degenerate.

The check that $S$ is a map of operads is a tedious, but straightforward calculation.
% We will check in the last section that $S$ is a map of equivariant chain complexes. The verification that $S$ commutes with operadic compositions is a more tedious calculation and will be omitted altogether. 
Let us just remark that the various terms of the formula \cite[2.26]{McSmith03} for the operadic composition in $\nices$ are hidden in the Eilenberg-Zilber transformation used to produce the operad $\fraka$ from the combinatorial operad $\{(\p_n,\partial\p_n)\}_{n\geq 0}$.
\end{itemize}

%=============================================================
%=============================================================
%=============================================================
\section{Proofs.}
\label{sect:appendix}
This section contains all the postponed proofs. They are expressed using the interpretation of morphisms in $\p$ as set maps $f$ with ordering on each fibre and rely on the notation and properties (\ref{eq1})-(\ref{eq4}) of Section \ref{subsect:extranotation}. We fix one more convention. Given a set $\ord{k_1+\cdots+k_n}$ the pair $(i,j)$ denotes the $j$-th position in the $i$-th summand, i.e. the element $k_1+\cdots+k_{i-1}+j$.

%=============================================================
\subsection{Proof that $\alpha$ is well-defined.}
\label{appendix:alphaok}
If $g\in\p(\ord{k'},\ord{k})$ then by (\ref{eqn:alpha}) and (\ref{eqn:joinstruct})
\begin{eqnarray*}
\alpha(f;x_1,\ldots,x_n)\circ g&=&(x_1\circ i_{f^{-1}(1)},\ldots,x_n\circ i_{f^{-1}(n)})\circ g\\
&=&(x_1\circ  i_{f^{-1}(1)}\circ g^{f^{-1}(1)},\ldots,x_n\circ  i_{f^{-1}(n)}\circ g^{f^{-1}(n)})
\end{eqnarray*}
while
\begin{eqnarray*}
\alpha((f;x_1,\ldots,x_n)\circ g)&=&\alpha(fg;x_1\circ g,\ldots,x_n\circ g)\\
&=&(x_1\circ  g\circ i_{(fg)^{-1}(1)},\ldots,x_n\circ g\circ i_{(fg)^{-1}(n)}).
\end{eqnarray*}
The two maps are equal because of (\ref{eq1}).

%=============================================================
\subsection{Proof of associativity of the join (Lemma \ref{lemma:associative}).}
\label{appendix:joinassoc}
The degree $k$ components of $J(J(X_{1,*}),\ldots,J(X_{n,*}))$ and $J(X_{*,*})$ are
\begin{equation}
\label{dd1}
\tag{$\ast$}
\coprod_{f\in\pp(\ord{k},\ord{n})}\ \prod_{i=1,\ldots,n}\ \coprod_{g_i\in\pp(\ord{\card{f^{-1}(i)}},\ord{k_i})}\ \prod_{j=1,\ldots,k_i}\ X_{i,j}(\card{g_i^{-1}(j)})
\end{equation}
and
\begin{equation}
\label{dd2}
\tag{$\ast\ast$}
\coprod_{h\in\pp(\ord{k},\ord{k_1+\cdots+k_n})}\ \prod_{(i,j)\in\ord{k_1+\cdots+k_n}}\ X_{i,j}(\card{h^{-1}(i,j)})
\end{equation}
respectively, where $\pp=\niceo$ for $J=\ojoin$ and $\pp=\sets$ for $J=\symjoin$.

The map $\Theta$ is defined as follows. The summand of (\ref{dd1}) corresponding to the choice $(f;g_1,\ldots,g_n)$ is being send to the summand of (\ref{dd2}) determined by $h=f\langle g_1,\ldots,g_n\rangle$ using on each $X_{i,j}$ the identity maps $X_{i,j}(\card{g_i^{-1}(j)})\xrightarrow{=} X_{i,j}(\card{h^{-1}(i,j)})$. The last statement makes sense since we have an equality of sets
$$h^{-1}(i,j)=g_i^{-1}(j).$$
It is a routine check, using (\ref{eq2}), that such defined maps combine to a map of $\niceo$-sets or $\p$-sets. 

The map $\Theta$ is an isomorphism because every map $h\in\pp(\ord{k},\ord{k_1+\cdots+k_n})$ has a unique decomposition of the form $h=f\langle g_1,\ldots,g_n\rangle$ where $f\in\pp(\ord{k},\ord{n})$ and $g_i\in\pp(\ord{\card{f^{-1}(i)}},\ord{k_i})$.

%Namely, $\phi$ is the composition of $\psi$ with the map sending all of $n_i$ to $i$ for $i=1,\ldots,m$ and $g_i=\psi^{n_i}$\footnote{It is worth noting that such a decomposition would not be unique for $\p$-morphisms.}
% Note that the uniqueness in the last statement holds only when $\pp=\niceo$ or $\pp=\sets$, but not for $\pp=\p$.

%=============================================================
\subsection{Proof of Proposition \ref{prop:atojoperadic}.}
\label{appendix:atojoperadic}
First note that the degree $k$ component of $A(X_*)$ can also be written as
$$A(X_*)(k)=\coprod_{f\in\p(\ord{k},\ord{n})} \prod_{i=1,\ldots,n} X_i(k).$$
With this convention the component in degree $k$ of the relevant diagram is:
\begin{adjustwidth}{-1in}{-1in}
\begin{minipage}[b]{1 \linewidth}\centering
\[
\xymatrix{
\underset{f\in\p(\ord{k},\ord{n})}{\coprod}\ \underset{i=1,\ldots,n}{\prod}\ \underset{g_i\in\p(\ord{k},\ord{k_i})}{\coprod}\ \underset{j=1,\ldots,k_i}{\prod}X_{i,j}(k)
\ar[r]^-{A(\alpha^n)}\ar[dd]_{\Psi_{k_1,\ldots,k_n}}
&
\underset{f\in\p(\ord{k},\ord{n})}{\coprod}\ \underset{i=1,\ldots,n}{\prod}\ \underset{g_i\in\sets(\ord{k},\ord{k_i})}{\coprod}\ \underset{j=1,\ldots,k_i}{\prod}X_{i,j}(\card{g_i^{-1}(j)})
\ar[d]^-\alpha
\\
&
\underset{f\in\sets(\ord{k},\ord{n})}{\coprod}\ \underset{i=1,\ldots,n}{\prod}\ \underset{g_i\in\sets(\ord{\card{f^{-1}(i)}},\ord{k_i})}{\coprod}\ \underset{j=1,\ldots,k_i}{\prod}X_{i,j}(\card{g_i^{-1}(j)})
\ar[d]^-\Theta
\\
\underset{h\in\p(\ord{k},\ord{k_1+\cdots+k_n})}{\coprod}\ \underset{(i,j)\in\ord{k_1+\cdots+k_n}}{\prod}\ X_{i,j}(k)
\ar[r]^-\alpha
& 
\underset{h\in\sets(\ord{k},\ord{k_1+\cdots+k_n})}{\coprod}\ \underset{(i,j)\in\ord{k_1+\cdots+k_n}}{\prod}\ X_{i,j}(\card{h^{-1}(i,j)})
}
\]
\end{minipage}
\end{adjustwidth}
We will use the formulas (\ref{map:phi2}), (\ref{eqn:alpha}), (\ref{eqn:joinstruct}) and the definition of $\Theta$ from \ref{appendix:joinassoc}.

Suppose we start in the summand of the upper-left corner indexed by the collection $(f;g_1,\ldots,g_n)$. Both ways of going around the diagram send this summand to the summand indexed by $$h=f\langle g_1\circ i_{f^{-1}(1)}, \ldots, g_n\circ i_{f^{-1}(n)}\rangle.$$
The map $\Theta\circ\alpha\circ A(\alpha^n)$ acts on each individual factor $X_{i,j}$ of that summand by the $\p$-morphism
$$\xi_1=i_{g_i^{-1}(j)}\circ(i_{f^{-1}(i)})^{g_i^{-1}(j)}\circ\id.$$
On the other hand, the map $\alpha\circ \Psi_{k_1,\ldots,k_n}$ acts on $X_{i,j}$ via the morphism
$$\xi_2=\id\circ i_{h^{-1}(i,j)}.$$
We have the equality
$$\xi_1=i_{g_i^{-1}(j)}\circ(i_{f^{-1}(i)})^{g_i^{-1}(j)}\circ\id=i_{f^{-1}(i)\cap g_i^{-1}(j)}= i_{h^{-1}(i,j)}=\xi_2$$
where the first transition follows from (\ref{eq4}) and the second from the equality of sets  $h^{-1}(i,j)=f^{-1}(i)\cap g_i^{-1}(j)$. This proves that the diagram commutes.

%=============================================================
\subsection{Proof of Theorem \ref{prop:psetoperad}.}
\label{appendix:psetoperad}

The operadic equivariance diagram involving the action of $\Sigma_n$ commutes because both the map $\alpha$ and the isomorphism $\Theta$ of (\ref{eqn:standard-assoc}) are $\Sigma_n$-equivariant. The equivariance with respect to the $(\Sigma_{k_1}\times\cdots\times\Sigma_{k_n})$-action follows from the fact that $\alpha$ is a natural transformation and because the isomorphism  (\ref{eqn:standard-assoc}) is $(\Sigma_{k_1}\times\cdots\times\Sigma_{k_n})$-equivariant.

Now we need to verify the operadic associativity axiom. Consider the family $X_{i,j}=\p_{k_{i,j}}$ of $\p$-sets as in Proposition \ref{prop:atojoperadic}, where $i=1,\ldots,n$ and $j=1,\ldots,k_i$ for some $k_i$. We need to show the commutativity of the outermost rectangle in the diagram
\begin{adjustwidth}{-1in}{-1in}
\begin{minipage}[b]{1 \linewidth}\centering
\begin{equation*}
\xymatrix{
A(A(\p_{k_{1,*}}),\ldots,A(\p_{k_{n,*}})) \ar[r]^-{A(\alpha^n)}\ar[dd]_{\Psi_{k_1,\ldots,k_n}} & A(\symjoin(\p_{k_{1,*}}),\ldots,\symjoin(\p_{k_{n,*}})) \ar[d]^-\alpha\ar[r]^-{A(\Theta^n)} &
A(\p_{\sum k_{1,*}},\ldots,\p_{\sum k_{n,*}}) \ar[d]^-\alpha\\
 & \symjoin(\symjoin(\p_{k_{1,*}}),\ldots,\symjoin(\p_{k_{n,*}}))\ar[d]^-{\Theta}\ar[r]^-{\symjoin(\Theta^n)} &
\symjoin(\p_{\sum k_{1,*}},\ldots,\p_{\sum k_{n,*}}) \ar[d]^-\Theta\\
A(\p_{k_{*,*}})\ar[r]^-\alpha & \symjoin(\p_{k_{*,*}}) \ar[r]^-\Theta& \p_{\sum k_{*,*}}
}
\end{equation*}
\end{minipage}
\end{adjustwidth}

The left rectangle commutes by Proposition \ref{prop:atojoperadic}. The top right square commutes because $\alpha:A\to\symjoin$ is a natural transformation. It remains to show the bottom right square commutes (note that all maps in that square are isomorphisms). The degree $k$ component of the iterated join $\symjoin(\symjoin(\p_{k_{1,*}}),\ldots,\symjoin(\p_{k_{n,*}}))$ is
\begin{equation*}
\coprod_{f\in\sets(\ord{k},\ord{n})}\ \prod_{i=1,\ldots,n}\ \coprod_{g_i\in\sets(\ord{\card{f^{-1}(i)}},\ord{k_i})}\ \prod_{j=1,\ldots,k_i}\ \p(\ord{\card{g_i^{-1}(j)}},\ord{k_{i,j}}).
\end{equation*}
Consider an element of this set given by a collection
$$f\in\sets(\ord{k},\ord{n}),\quad \big\{g_i\in\sets(\ord{\card{f^{-1}(i)}},\ord{k_i})\big\}_{i=1,\ldots,n},$$ $$\big\{h_{i,j}\in\p(\ord{\card{g_i^{-1}(j)}},\ord{k_{i,j}})\big\}_{(i,j)\in\ord{k_1+\cdots+k_n}}.$$
The two ways of going to $\p_{\sum k_{*,*}}(k)$ send this family, respectively, to the morphisms
$$\big(\ f\langle g_1,\ldots,g_n\rangle\ \big)\langle h_{1,1},\ldots,h_{1,k_1},\ldots,h_{n,1},\ldots,h_{n,k_n}\rangle$$
and
$$f\langle\ g_1\langle h_{1,1},\ldots,h_{1,k_1}\rangle,\ldots,g_n\langle h_{n,1},\ldots,h_{n,k_n}\rangle\ \rangle$$
% $$\big(\ f\langle g_i\rangle_{i=1}^n\ \big)\langle h_{i,j}\rangle_{(i,j)\in\ord{k_1+\cdots+k_n}},\quad f\langle\ g_i\langle h_{i,j}\rangle_{j=1}^{k_i}\ \rangle_{i=1}^n$$
which are easily seen to be equal as elements of $\p(\ord{k},\ord{\sum k_{*,*}})$.

%=============================================================
\subsection{Proof of Proposition \ref{pre-operad}.}
\label{appendix:yetanotherdiagram}
Each of the maps used in the definition of $\fraka_n^X$ is $\Sigma_n$-equivariant, hence so is the composition. For $C$ to be a coalgebra over $\fraka$ we need the commutativity of the diagram
\begin{equation}
\label{coalg-coherence}
\xymatrix{
\fraka(n)\otimes\fraka(k_1)\otimes\cdots\otimes\fraka(k_n)\otimes C \ar[rr]\ar[d] & & \fraka(k_1)\otimes\cdots\otimes\fraka(k_n)\otimes C^{\otimes n} \ar[d]\\
\fraka(\sum k_*)\otimes C \ar[rd] &  & \fraka(k_1)\otimes C\otimes\cdots\fraka(k_1)\otimes C \ar[ld] \\
& C^{\otimes \sum k_*} &
}
\end{equation}

For this we first generalize the combinatorial `coaction' maps (\ref{osigma-coaction-nonrel}). For any sequence $X_1,\ldots,X_n$ of $\p$-sets and integers $k_1,\ldots,k_n$ we have a map
\begin{equation}
\label{general-map}
\p_{k_1}\times\cdots\times\p_{k_n}\times\symjoin(X_1,\ldots,X_n)\to\symjoin(\p_{k_1}\times X_1,\ldots,\p_{k_n}\times X_n).
\end{equation}
On components of degree $k$
$$\p(\ord{k},\ord{k_1})\times\cdots\times\p(\ord{k},\ord{k_n})\times\coprod_{f\in\sets(\ord{k},\ord{n})}\ \prod_{i=1}^n X_i(\card{f^{-1}(i)})\to \coprod_{f\in\sets(\ord{k},\ord{n})}\ \prod_{i=1}^n \p(\ord{\card{f^{-1}(i)}},\ord{k_i})\times X_i(\card{f^{-1}(i)})$$
it is given by the formula
$$(g_1,\ldots,g_n;x_1,\ldots,x_n)\to((g_1\circ i_{f^{-1}(1)},x_1),\ldots,(g_n\circ i_{f^{-1}(n)},x_n)).$$
By post-composing the map of (\ref{general-map}) with the `coactions' $\psi_{k_i}^{X_i}$ of (\ref{osigma-coaction-nonrel}), we obtain maps
\begin{equation}
\label{general-map2}
\widetilde{\psi}:\p_{k_1}\times\cdots\times\p_{k_n}\times\symjoin(X_1,\ldots,X_n)\to\symjoin(\symjoin(X_1^{k_1}),\ldots,\symjoin(X_n^{k_n}))
\end{equation}
which satisfy the following compatibility.

\begin{lemma}
\label{new-diagram}
For any $\p$-set $X$ and integers $n,k_1,\ldots,k_n$ the following diagram of $\p$-sets commutes.
\begin{equation*}
\xymatrix{
\p_n\times\p_{k_1}\times\cdots\times\p_{k_n}\times X \ar[rr]^{\psi_n^X} \ar[d]_{\psi_{k_1,\ldots,k_n}} & & \p_{k_1}\times\cdots\times\p_{k_n}\times \symjoin(X^n)\ar[d]^{\widetilde{\psi}}\\
\p_{\sum k_*}\times X \ar[rd]_{\psi_{\sum k_*}^X} &  & \symjoin(\symjoin(X^{k_1}),\ldots,\symjoin(X^{k_n})) \ar[ld]^{\Theta} \\
 & \symjoin(X^{\sum k_*}) &
}
\end{equation*}
\end{lemma}
\begin{proof}
The proof resembles that of Section \ref{appendix:atojoperadic} and uses (\ref{map:phi2}), (\ref{eqn:alpha}), (\ref{eqn:joinstruct}), the map $\Theta$ from \ref{appendix:joinassoc} and (\ref{general-map2}). Consider a tuple $(f;g_1,\ldots,g_n;x)$ in $\p_n\times\p_{k_1}\times\cdots\times\p_{k_n}\times X$. Each of the two ways around the diagram sends it to the summand of the join indexed by the map
$$h=f\langle g_1\circ i_{f^{-1}(1)}, \ldots, g_n\circ i_{f^{-1}(n)}\rangle.$$
The result in $\symjoin(X^{\sum k_*})$ is a tuple indexed by pairs $(i,j)$, $1\leq i\leq n$, $1\leq j\leq k_i$. The left path in the diagram produces a tuple whose $(i,j)$-th entry is
$$x\circ i_{h^{-1}(i,j)}$$
while for the other path it is
$$x\circ i_{f^{-1}(i)}\circ i_{(g_i\circ i_{f^{-1}(i)})^{-1}(j)}=x\circ i_{f^{-1}(i)}\circ i_{(i_{f^{-1}(i)})^{-1}(g_i^{-1}(j))}=x\circ i_{g_i^{-1}(j)}\circ(i_{f^{-1}(i)})^{g_i^{-1}(j)}$$
where the last equality follows from (\ref{eq1}). We conclude that the two maps are equal as in Section \ref{appendix:atojoperadic}.
\end{proof}

To complete the proof of Proposition \ref{pre-operad} note that Lemma \ref{new-diagram} gives, after application of the chains functor $C_\ast^\niceo(\cdot)$, the following commutative diagram for every $\niceo$-set $X$.

\begin{adjustwidth}{-1in}{-1in}
\begin{minipage}[b]{1 \linewidth}\centering
\begin{equation*}
\xymatrix{
C_\ast^\niceo(\p_n)\otimes C_\ast^\niceo(\p_{k_1})\otimes\cdots\otimes C_\ast^\niceo(\p_{k_n})\otimes C_\ast^\niceo(X) \ar[d]^{C_\ast^\niceo(\eta_X)} & & \\
C_\ast^\niceo(\p_n)\otimes C_\ast^\niceo(\p_{k_1})\otimes\cdots\otimes C_\ast^\niceo(\p_{k_n})\otimes C_\ast^\niceo(X\Sigma) \ar[d]^{EZ} & & \\
C_*^\niceo(\p_n\times\p_{k_1}\times\cdots\times\p_{k_n}\times X\Sigma) \ar[rr] \ar[d] & & C_*^\niceo(\p_{k_1}\times\cdots\times\p_{k_n}\times \symjoin(X\Sigma^n))\ar[d] \\
C_*^\niceo(\p_{\sum k_*}\times X\Sigma) \ar[rd] &  & C_*^\niceo(\symjoin(\symjoin(X\Sigma^{k_1}),\ldots,\symjoin(X\Sigma^{k_n}))) \ar[ld] \\
 & C_*^\niceo(\symjoin(X\Sigma^{\sum k_*})) \ar[d]^{s} & \\
& C_\ast^\niceo(X)^{\otimes \Sigma k_*} & 
}
\end{equation*}
\end{minipage}
\end{adjustwidth}
where $s$ is the map of Proposition \ref{strange:equivariance}.

It follows that the same diagram commutes for the relative objects $(\p,\partial\p)$, $(X,X(0))$ and $(X\Sigma,X\Sigma(0))$ and the two ways of traversing that diagram correspond to the two ways around (\ref{coalg-coherence}) for $C=C_*^\niceo(X,X(0))$. That ends the proof.

\end{document}